\documentclass[12pt,twoside]{amsart}
\usepackage{amsmath,amsthm,amsmath,amsrefs}

\usepackage{calc}
\usepackage{setspace}
\usepackage{amsbsy,amsmath,amsfonts,amsthm,amssymb,color}
\usepackage[margin=1in]{geometry}
\theoremstyle{plain}
\newtheorem{mythe}{Theorem} %[section]
\newtheorem{lem}[mythe]{Lemma}

\newtheorem{cor}[mythe]{Corollary}

\newtheorem{pro}[mythe]{Proposition}

\theoremstyle{definition}

\usepackage{mathrsfs}

\newcommand{\ee}{\varepsilon}

\newcommand{\cB}{\mathcal{B}}
\newcommand{\bT}{\mathbb{T}}
\newcommand{\cK}{\mathcal{K}}

\newcommand{\la}{\langle}
\newcommand{\ra}{\rangle}

\newcommand{\cH}{\mathcal{H}}
\newcommand{\cA}{\mathcal{A}}

\newcommand{\bofh}{\cB(\cH)}

\newcommand{\bC}{\mathbb{C}}

\newcommand{\spn}{\text{span }}

\usepackage{tikz}
\usepackage{tikz-cd}

\title[Connes' Embedding Problem and Quantum XOR Games]{Connes' Embedding Problem and Winning Strategies for Quantum XOR Games}

\author{Samuel J. Harris}
\address{University of Waterloo
Department of Pure Mathematics 
Waterloo, Ontario 
Canada  N2L 3G1}

\email{sj2harri@uwaterloo.ca}

\begin{document}
\begin{abstract}
We consider quantum XOR games, defined in \cite{RV}, from the perspective of unitary correlations defined in \cite{HP}.  We show that Connes' embedding problem has a positive answer if and only if every quantum XOR game has entanglement bias equal to the commuting bias.  In particular, the embedding problem is equivalent to determining whether every quantum XOR game $G$ with a winning strategy in the commuting model also has a winning strategy in the approximate finite-dimensional model.
\end{abstract}

\maketitle

\section{Introduction}

A long-standing open problem in the theory of von Neumann algebras, known as Connes' embedding problem \cite{connes}, asks whether every weakly separable $II_1$ factor can be approximately embedded into the hyperfinite $II_1$ factor in a way that approximately preserves the trace.  Recently, several connections between the embedding problem and questions in quantum information theory have been exhibited.  The most notable connection is the equivalence of Connes' embedding problem to the weak Tsirelson problem regarding probabilistic correlations in a separated system \cite{junge,fritz,ozawa}.  This problem asks whether every probabilistic quantum bipartite correlation in finite inputs and finite outputs in the commuting model can be approximated by correlations in the finite-dimensional model \cite{tsirelson93}.  

These quantum bipartite correlations naturally correspond to strategies that can be used in two-player, finite input-output non-local games.  Such games have been instrumental in exhibiting separations between the probabilistic correlation sets in the various models.  For example, the well-known CHSH game \cite{CHSH} is a famous example of a game for which the maximum winning probability of the game for Alice and Bob is higher if they have access to entanglement than if they play the game using classical methods.  Recent work by W. Slofstra \cite{slofstra17} shows that there is a non-local game with a winning strategy in the approximate finite-dimensional model, but no winning strategy in the finite-dimensional model.  Therefore, there are input and output sets for which the set of correlations arising from the finite-dimensional model is not closed.  Considering these advances, a natural question is whether the weak Tsirelson problem can be described in terms of a certain class of (possibly extended) non-local games and winning strategies in the commuting model and the approximate finite-dimensional model.  In other words, is Connes' embedding problem equivalent to the assertion that a certain class of extended non-local games with winning commuting strategies must also have winning approximate finite-dimensional strategies?

In this paper, we show that the answer to this question is affirmative, and that the class of quantum XOR games (defined in \cite{RV}) is rich enough to detect the solution to the embedding problem.  In particular, determining whether every quantum XOR game with winning commuting strategy has a winning approximate finite-dimensional strategy is equivalent to Connes' embedding problem.  If one considers the analogous weak Tsirelson problem related to unitary correlation sets \cite{HP}, the coherent embezzlement games from \cite{RV} can be used to show that the set of unitary correlations in the finite-dimensional model are not closed \cite{CLP,HP}, as long as the unitaries have size at least $2$.  In light of these facts, it is plausible that studying quantum XOR games may be a reasonable plan of attack for solving the embedding problem.

The paper is organized as follows.  In Section 2, we give a brief overview of quantum XOR games from \cite{RV} and the notion of bias for these games.  We also show the correspondence between bias and linear functionals on $M_n \otimes M_n$ that are contractive with respect to the unitary correlation norms from \cite{HP}.  In Section 3, we use Lemma \ref{selfadjoint} to reduce the Tsirelson problem for unitary correlations to self-adjoint unitary correlations.  This allows us to prove Corollary \ref{winning}, which states the equivalence of the embedding problem to the problem of winning strategies for quantum XOR games in the commuting and approximate finite-dimensional models.

\section{Preliminaries}

We briefly recall the definition of a quantum XOR game with two parties, Alice and Bob.  More information can be found in \cite{RV}.  Loosely speaking, a quantum XOR game is a generalization of a classical XOR game.  In the classical case, the referee has a list of $n$ possible questions $\{1,...,n\}$, and the set of possible answers for Alice and Bob is $\{0,1\}$.  For each pair $s,t \in \{1,...,n\}$, there is some associated number $R_{s,t} \in [-1,1]$ (known to Alice and Bob) satisfying $\sum_{s,t} |R_{s,t}|=1$.  The referee gives question $s$ to Alice and question $t$ to Bob with probability $|R_{s,t}|$.  If $R_{s,t} \geq 0$, then Alice and Bob must respond with the same bit; if $R_{s,t}<0$, then they must respond with different bits.  

In a quantum XOR game, the questions are now given as states (i.e., unit vectors) on a certain Hilbert space. In particular, the referee sends some state on $\bC^n \otimes \bC^n$ to Alice and Bob, where Alice has access to the left copy of $\bC^n$ and Bob has access to the right copy of $\bC^n$.  Every quantum XOR game of size $n$ is associated with a self-adjoint matrix $M \in M_n \otimes M_n$ with $\|M\|_1 \leq 1$, where $\| \cdot \|_1$ denotes the trace norm.  Conversely, every self-adjoint matrix $M \in M_n \otimes M_n$ with $\|M\|_1 \leq 1$ is associated to a quantum XOR game $G$ of size $n$ \cite{RV}.  For the sake of simplicity, we will always consider the case where $\|M\|_1=1$.

For our purposes, a quantum XOR game $G$ of size $n$ can be described as follows (see \cite{RV} for a more general definition): let $\{ \varphi_i\}_{i=1}^{n^2} \subseteq \bC^n \otimes \bC^n$ be an orthonormal basis.  Let $p_1,...,p_{n^2} \in [0,1]$ be such that $\sum_{i=1}^{n^2} p_i=1$, and let $c_i \in \{0,1\}$ for each $1 \leq i \leq n^2$.  With probability $p_i$, the referee prepares the state $\varphi_i \in \bC^n \otimes \bC^n$.  Alice and Bob may use their own Hilbert space $\cH$ and \textbf{observables} $A \in \cB(\bC^n \otimes \cH)$ and $B \in \cB(\cH \otimes \bC^n)$ (i.e., self-adjoint unitaries) such that $A \otimes I_n$ and $I_n \otimes B$ commute in $\cB(\bC^n \otimes \cH \otimes \bC^n)$.  They may also prepare their space in the state $\psi \in \cH$.  Based on the application of $A$ and $B$ to the state $\psi$, Alice and Bob return outcomes $a \in \{0,1\}$ and $b \in \{0,1\}$ respectively.  If $c_i=0$, then Alice and Bob's output bits must be equal; if $c_i=1$, their output bits must be distinct.  (If Alice and Bob are working in the tensor product model, then there must be a decomposition $\cH=\cH_A \otimes \cH_B$ where $A \in \cB(\bC^n \otimes \cH_A)$ and $B \in \cB(\cH_B \otimes \bC^n)$, and where $\psi \in \cH_A \otimes \cH_B$ is a state.  Moreover, the operator $(A \otimes I_n)(I_n \otimes B)$ is replaced with $A \otimes B$.)  The matrix $M \in M_n \otimes M_n$ associated with this quantum XOR game is $M=\sum_{i=1}^{n^2} (-1)^{c_i} p_i \varphi_i \varphi_i^*$, where $\varphi_i \varphi_i^*$ denotes the rank one orthogonal projection of $\bC^n \otimes \bC^n$ onto $\spn \{\varphi_i\}$ \cite{RV}.

Before we further consider possible strategies for Alice and Bob for quantum XOR games, we recall the definitions of some of the unitary correlation sets given in \cite{HP}.  The set of \textbf{quantum correlations} $B_q(n,n) \subseteq M_n \otimes M_n$ is given by the set of all $X=(X_{(i,j),(k,\ell)}) \in M_n \otimes M_n$ of the form $$X_{(i,j),(k,\ell)}=\la (U_{ij} \otimes V_{k\ell})\psi,\psi \ra,$$
where $U=(U_{ij}) \in M_n(\cB(\cH_A))$ and $V=(V_{k\ell}) \in M_n(\cB(\cH_B))$ are unitary, $\cH_A$ and $\cH_B$ are finite-dimensional Hilbert spaces, and $\psi \in \cH_A \otimes \cH_B$ is a unit vector.  The set of \textbf{quantum spatial correlations} $B_{qs}(n,n) \subseteq M_n \otimes M_n$ is defined similarly, only that we no longer assume that $\cH_A$ and $\cH_B$ are finite-dimensional.  The set of \textbf{quantum approximate correlations} is given by $B_{qa}(n,n)=\overline{B_q(n,n)}=\overline{B_{qs}(n,n)}$.  The set of \textbf{quantum commuting correlations} is given by the set of all $X \in M_n \otimes M_n$ of the form $$X=(\la U_{ij}V_{k\ell}\psi,\psi \ra)_{(i,j),(k,\ell)},$$
where $U=(U_{ij}) \in M_n(\bofh)$ and $V=(V_{k\ell}) \in M_n(\bofh)$ are unitary, $\psi \in \cH$ is a unit vector, and $U_{ij}V_{k\ell}=V_{k\ell}U_{ij}$ for all $i,j,k,\ell$.  Since $U$ and $V$ are unitary, it also follows that the set $\{U_{ij},U_{ij}^*\}_{i,j=1}^n$ commutes with the set $\{V_{k\ell},V_{k\ell}^*\}_{k,\ell=1}^n$.

We have the containments $$B_q(n,n) \subseteq B_{qs}(n,n) \subseteq B_{qa}(n,n) \subseteq B_{qc}(n,n), \, \forall n \geq 2,$$
and the latter two sets are compact in $M_n \otimes M_n$ \cite{HP}.  In fact, for $t \in \{qa,qc\}$, the set $B_t(n,n)$ is the unit ball of a reasonable cross-norm on $M_n \otimes M_n$ \cite[Theorem 4.8]{HP}; we will denote this norm by $\| \cdot \|_t$, and we will denote the Banach space $(M_n \otimes M_n,\| \cdot \|_t)$ by $M_n \otimes_t M_n$.  Finally, we will let $\| \cdot\|_t^*$ denote the dual norm of $\| \cdot \|_t$ on $(M_n \otimes M_n)^*$.

For a quantum XOR game $G$ and $t \in \{q,qa,qc\}$, we define a \textbf{$t$-strategy} for Alice and Bob to be a correlation $X \in B_t(n,n)$.

Instead of working with maximum success probabilities in different models, it is convenient to work with a related quantity, known as the bias.  To ease notation, whenever $\cH$ and $\cK$ are Hilbert spaces and $\cH$ is finite-dimensional, we will denote by $\text{Tr}_{\cH}$ the operator $\text{Tr} \otimes \text{id}_{\cK}$ acting on $\cB(\cH \otimes \cK)$, where $\text{Tr}$ denotes the unnormalized trace on $\cB(\cH)$.  With this in hand, the \textbf{entanglement bias} (or the \textbf{quantum bias}) of a quantum XOR game $G$ with associated matrix $M$ is given by $$\omega_q^*(G)=\sup \{ \la \text{Tr}_{\bC^n \otimes \bC^n}[(A \otimes B)(M \otimes I_{\cH_A \otimes \cH_B})] \psi,\psi \ra \},$$
where the supremum is taken over all finite-dimensional Hilbert spaces $\cH_A$ and $\cH_B$, unit vectors $\psi \in \cH_A \otimes \cH_B$, and observables $A \in \cB(\bC^n \otimes \cH_A)$ and $B \in \cB(\cH_B \otimes \bC^n)$ (that is, self-adjoint unitaries) \cite{RV}. In the supremum above, we are identifying $M \otimes I_{\cH_A \otimes \cH_B}$ with the matrix $M' \in M_{n^2}(\cB(\cH_A \otimes \cH_B))$ given by $M'=(M_{(i,j),(k,\ell)} I_{\cH_A \otimes \cH_B})_{(i,j),(k,\ell)}$.

Similarly, we define the \textbf{commuting bias} of a quantum XOR game $G$ with associated matrix $M$ to be $$\omega_{qc}^*(G)=\sup \{ \la \text{Tr}_{\bC^n \otimes \bC^n}[((A \otimes I_{\bC^n})(I_{\bC^n} \otimes B)(M \otimes I_{\cH})]\psi,\psi \ra,$$
where the supremum is taken over all Hilbert spaces $\cH$, unit vectors $\psi \in \cH$, and self-adjoint unitaries $A=(A_{ij}) \in \cB(\bC^n \otimes \cH)$ and $B=(B_{k\ell}) \in \cB(\cH \otimes \bC^n)$ such that $(A \otimes I_n)(I_n \otimes B)=(I_n \otimes B)(A \otimes I_n)$ as operators on $\bC^n \otimes \cH \otimes \bC^n$.  Here, we are identifying $M \otimes I_{\cH}$ with the operator $M' \in M_{n^2}(\bofh)$ given by $M'=(M_{(i,j),(k,\ell)} I_{\cH})$.  Adapting the proof of \cite[Proposition 3.1]{CLP}, since the matrix $(A \otimes I_n)(I_n \otimes B)$ is given by $(A_{ij}B_{k\ell})_{(i,j),(k,\ell)}$, it follows that $A \otimes I_n$ commutes with $I_n \otimes B$ if and only if $A_{ij}B_{k\ell}=B_{k\ell}A_{ij}$ for all $i,j,k,\ell$.  In particular, since $A=A^*$ and $B=B^*$, we have $A_{ij}=A_{ji}^*$ and $B_{k\ell}=B_{\ell k}^*$.  Thus, the self-adjoint unitaries $A$ and $B$ satisfy $(A \otimes I_n)(I_n \otimes B)=(I_n \otimes B)(A \otimes I_n)$ if and only if the set $\{I\} \cup \{A_{ij}\}_{i,j=1}^n$ $*$-commutes with the set $\{I\} \cup \{B_{k\ell}\}_{k,\ell=1}^n$.  

One may view both of these notions of bias as twice the difference between the maximum success probability in the corresponding model and the success probability from the random strategy (i.e. Alice and Bob respond randomly, regardless of the input).  (The argument from \cite{RV} for entanglement bias can be adapted in the obvious way for the commuting bias.)  Thus, if $p$ is the maximum success probability of winning a quantum XOR game $G$ using $t$-strategies (for $t \in \{q,qc\}$), then $\omega_t^*(G)=2p-1$.

We will also consider the above definitions of bias that arise from omitting the assumption that $A$ and $B$ are self-adjoint, while keeping the other assumptions intact.  In particular, we may consider the notions of bias given with respect to the unitary correlation sets defined above.  In this context, we also consider the bias of a particular strategy, which has the same definition but is denoted by $\omega_t^*(G,U,V,\psi)$ for a specific strategy $(U,V,\psi)$ whose correlation matrix $X$ is in $B_t(n,n)$.  If $X$ is the correlation associated with $(U,V,\psi)$, then we also let $\omega_t^*(G,X)=\omega_t^*(G,U,V,\psi)$.  We note that $\omega_t^*(G,U,V,\psi)$ is $\bC$-valued.  A \textbf{perfect $t$-strategy} is a $t$-strategy $X$ for which $\omega_t^*(G,X)=1$.  It will follow from Theorem \ref{biasunitaries} that there is a perfect $t$-strategy for the quantum XOR game $G$ if and only if there is a $t$-strategy arising from observables for which the probability that Alice and Bob win is $1$.

We remark that there is a natural correspondence between bias for quantum XOR games and self-adjoint linear functionals on $M_n \otimes M_n$.  The easiest way to see this is using unitary correlation sets.  Indeed, suppose that $X=(\la U_{ij}V_{k\ell}\psi,\psi \ra)_{(i,j),(k,\ell)} \in B_{qc}(n,n)$.  Then since $M$ is self-adjoint,
\begin{align*}
\omega_{qc}^*(G,X)&=\la \text{Tr}_{\bC^n \otimes \bC^n}[(U_{ij}V_{k\ell})_{(i,j),(k,\ell)} (M_{(i,j),(k,\ell)}I_{\cH})_{(i,j),(k,\ell)}] \psi,\psi \ra \\
&=\left\la \text{Tr}_{\bC^n \otimes \bC^n} \left( \sum_{p,q=1}^n U_{ip}V_{kq} M_{(p,j),(q,\ell)} \right) \psi,\psi \right\ra \\
&=\sum_{i,j,p,q} \la U_{ip}V_{jq}M_{(p,i),(q,j)} \psi,\psi \ra \\
&=\sum_{i,j,p,q} M_{(p,i),(q,j)} \la U_{ip}V_{jq}\psi,\psi \ra \\
&=\sum_{i,j,k,\ell} M_{(j,i),(\ell,k)} X_{(i,j),(k,\ell)}, \\
&=\text{Tr}(MX).
\end{align*}
Thus, the quantum XOR game $G$ defines the linear functional $\text{Tr}(M \cdot):M_n \otimes M_n \to \bC$ which gives $\omega_{qc}^*(G,X)$ for each $X \in B_{qc}(n,n)$.  An analogous argument holds for $\omega_q^*(G,X)$ whenever the corresponding correlation matrix $X$ lies in $B_q(n,n)$.

A helpful fact is that for $t \in \{q,qs,qa,qc\}$, the self-adjoint $t$-correlations arise from $t$-strategies involving observables.

\begin{pro}
\label{selfadjointobservable}
Let $t \in \{q,qs,qa,qc\}$ and $X=X^* \in B_t(n,n)$.  Then there is a Hilbert space $\cH$, self-adjoint unitaries $U \in \cB(\bC^n \otimes \cH)$ and $V \in \cB(\cH \otimes \bC^n)$ and a unit vector $\psi \in \cH$ such that $X_{(i,j),(k,\ell)}=\la U_{ij}V_{k\ell}\psi,\psi \ra$ for all $i,j,k,\ell$.  If $t=qs$, then we may take $\cH=\cH_A \otimes \cH_B$ and $U \in \cB(\bC^n \otimes \cH_A)$ and $V \in \cB(\cH_B \otimes \bC^n)$ such that $X_{(i,j),(k,\ell)}=\la (U_{ij} \otimes V_{k\ell})\psi,\psi \ra$ for all $i,j,k,\ell$.  Moreover, if $t=q$, then we may take $\cH_A$ and $\cH_B$ to be finite-dimensional. Finally, if $t=qa$, then $X=\lim_{m \to \infty} Y^{(m)}$, where $Y^{(m)} \in B_q(n,n)$ is a $q$-strategy involving observables.
\end{pro}

\begin{proof}
We first let $t=q$, so that $X=(\la (R_{ij} \otimes S_{k\ell})\psi,\psi \ra)_{(i,j),(k,\ell)}$ for unitaries $R=(R_{ij}) \in \cB(\bC^n \otimes \cH_A)$, $S=(S_{k\ell}) \in \cB(\cH_B \otimes \bC^n)$, finite-dimensional Hilbert spaces $\cH_A$ and $\cH_B$, and a unit vector $\psi \in \cH_A \otimes \cH_B$.  Let $U_{ij}=\begin{bmatrix} 0 & R_{ij} \\ R_{ji}^* & 0 \end{bmatrix}$ and $V_{k\ell}=\begin{bmatrix} 0 & S_{k\ell} \\ S_{\ell k}^* & 0 \end{bmatrix}$.  Performing a canonical shuffle \cite[p.~97]{paulsen02} on the unitaries $\begin{bmatrix} 0 & R \\ R^* & 0 \end{bmatrix} \in M_2(M_n(\cB(\cH_A)))$ and $\begin{bmatrix} 0 & S \\ S^* & 0 \end{bmatrix} \in M_2(M_n(\cB(\cH_B)))$, we see that $U=(U_{ij})$ and $V=(V_{k\ell})$ are self-adjoint unitaries in $M_n(\cB(\cH_A \oplus \cH_A))$ and $M_n(\cB(\cH_B \oplus \cH_B))$, respectively.  Hence, we obtain a $q$-strategy $(U,V,\widetilde{\psi})$, where $\widetilde{\psi}=\frac{1}{\sqrt{2}}\begin{bmatrix} \psi & 0 & 0 & \psi \end{bmatrix}^t$.  Using the fact that $U_{ij}^*=U_{ji}$ and $V_{k\ell}^*=V_{\ell k}$, $$\la (U_{ij} \otimes V_{k\ell})\widetilde{\psi},\widetilde{\psi} \ra=\frac{1}{2} \la (R_{ij} \otimes S_{k\ell})\psi,\psi \ra+\frac{1}{2} \la (R_{ji}^* \otimes S_{\ell k}^*) \psi,\psi \ra=X_{(i,j),(k,\ell)}.$$
The proof for $t=qs$ is similar.
For $t=qc$, we assume that $R \in \cB(\bC^n \otimes \cH)$ and $S \in \cB(\cH \otimes \bC^n)$ are unitaries and $\psi \in \cH$ is a unit vector such that $(R \otimes I_n)(I_n \otimes S)=(I_n \otimes S)(R \otimes I_n)$ and $X_{(i,j),(k,\ell)}=\la R_{ij}S_{k\ell}\psi,\psi \ra$ for all $i,j,k,\ell$.  We let $$U_{ij}=\begin{bmatrix} 0 & R_{ij} & 0 & 0 \\ R_{ji}^* & 0 & 0 & 0 \\ 0 & 0 & 0 & R_{ij} \\ 0 & 0 & R_{ji}^* & 0 \end{bmatrix} \text{ and } V_{k\ell}=\begin{bmatrix} 0 & 0 & S_{k\ell} & 0 \\ 0 & 0 & 0 & S_{k\ell} \\ S_{\ell k}^* & 0 & 0 & 0 \\ 0 & S_{\ell k}^*  & 0 & 0 \end{bmatrix}.$$
A calculation shows that
$$U_{ij}V_{k\ell}=\begin{bmatrix} 0 & 0 & 0 & R_{ij}S_{k\ell} \\ 0 & 0 & R_{ji}^*S_{k\ell} & 0 \\ 0 & R_{ij}S_{\ell k}^* & 0 & 0 \\ R_{ji}^*S_{\ell k}^* & 0 & 0 & 0 \end{bmatrix}=V_{k\ell}U_{ij},$$
so that $U=(U_{ij}) \in \cB(\bC^n \otimes \cH)$ and $V=(V_{k\ell}) \in \cB(\cH \otimes \bC^n)$ are self-adjoint unitaries with $(U \otimes I_n)(I_n \otimes V)=(I_n \otimes V)(U \otimes I_n)$.  Letting $\widetilde{\psi}=\frac{1}{\sqrt{2}} \begin{bmatrix} \psi & 0 & 0 & \psi \end{bmatrix}^t$, it readily follows that $$\la U_{ij}V_{k\ell}\psi,\psi \ra=\frac{1}{2} \la R_{ij}S_{k\ell}\psi,\psi \ra+\frac{1}{2}\la R_{ji}^*S_{\ell k}^*\psi,\psi \ra=X_{(i,j),(k,\ell)}.$$
Thus, the proposition holds for $t=qc$.  The last statement about $qa$ correlations immediately follows from the $t=q$ case.
\end{proof}

The next proposition, combined with convexity of each $B_t(n,n)$ \cite{HP}, guarantees that $\text{Re}(X) \in B_t(n,n)$ whenever $X \in B_t(n,n)$.

\begin{pro}
\label{starnorm}
Let $n \geq 2$ and $X=(X_{(i,j),(k,\ell)})_{(i,j),(k,\ell)} \in B_t(n,n)$ for $t \in \{q,qs,qa,qc\}$.  Then $X^* \in B_t(n,n)$.
\end{pro}

\begin{proof}
Suppose that $U=(U_{ij}) \in M_n(\cB(\cH))$ and $V=(V_{k\ell}) \in M_n(\bofh)$ are unitary and $\psi \in \cH$ is a unit vector such that $U_{ij}V_{k\ell}=V_{k\ell}U_{ij}$ for all $i,j,k,\ell$ and $$\la U_{ij}V_{k\ell}\psi,\psi \ra=X_{(i,j),(k,\ell)}.$$
Then $$X^*=(\overline{X}_{(j,i),(\ell,k)})=(\overline{\la U_{ji}V_{\ell k}\psi,\psi \ra})=(\la \psi,U_{ji}V_{\ell k} \psi \ra)=(\la U_{ji}^*V_{\ell k}^* \psi,\psi \ra),$$
using the fact that $U_{ij}^*V_{k\ell}^*=V_{k\ell}^*U_{ij}^*$ for all $i,j,k,\ell$.  It follows that $$X^*=(\la (U^*)_{ij} (V^*)_{k\ell} \psi,\psi \ra)_{(i,j),(k,\ell)} \in B_{qc}(n,n).$$
A similar argument gives the desired result when $t \in \{q,qs\}$.  The same result follows for $t=qa$ by using the density of $B_q(n,n)$ in $B_{qa}(n,n)$.
\end{proof}

\begin{cor}
For $n \geq 2$ and $t \in \{loc,q,qs,qa,qc\}$, let $\| \cdot \|_t^*$ denote the dual norm on $M_n \otimes M_n$ with respect to the normed space $M_n \otimes_t M_n$.  Then $\| \cdot \|_t$ is a $*$-norm; i.e., if $\|M\|_t^* \leq 1$, then $\|M^*\|_t^* \leq 1$.
\end{cor}

\begin{proof}
Let $M \in M_n \otimes M_n$ be such that $\|M\|_t \leq 1$.  The linear functional on $M_n \otimes M_n$ with respect to $M$ is given by $$f((X_{(i,j),(k,\ell)}))=\text{Tr}(XM).$$
If $g$ is the linear functional on $M_n \otimes M_n$ with respect to $M^*$, then $$g((X_{(i,j),(k,\ell)}))=\text{Tr}(XM^*)=\overline{\text{Tr}(MX^*)}=\overline{\text{Tr}(X^*M)}=\overline{f(X^*)}.$$
Since $\| \cdot \|_t$ is a $*$-norm, it follows that $\|M^*\|_t \leq 1$, as required.
\end{proof}

Using Propositions \ref{selfadjointobservable} and \ref{starnorm}, we obtain an equivalent description of bias, which allows us to use the theory of unitary correlations.

\begin{mythe}
\label{biasunitaries}
Let $G$ be a quantum XOR game of size $n$ with associated matrix $M \in M_n \otimes M_n$.  Then $$\omega_{qc}^*(G)=\sup \{ |\text{Tr}(MX)|: X \in B_{qc}(n,n)\}.$$
Similarly, we have $$\omega_q^*(G)=\sup \{ |\text{Tr}(MX)|: X \in B_q(n,n)\}.$$
\end{mythe}

\begin{proof}
Since every observable is a self-adjoint unitary, it is clear that $\omega_{qc}^*(G)$ is at least the quantity given in the theorem statement; thus, we need only establish the reverse inequality. Using the fact that $B_{qc}(n,n)$ is compact \cite{HP}, we may choose $X \in B_{qc}(n,n)$ such that $$\sup \{ |\text{Tr}(MZ)|: Z \in B_{qc}(n,n)\}=\text{Tr}(MX).$$
Suppose that $X=\la U_{ij}V_{k\ell}\psi,\psi \ra$ where $U=(U_{ij}) \in \cB(\bC^n \otimes \cH)$ and $V=(V_{k\ell}) \in \cB(\cH \otimes \bC^n)$ are unitaries, $\psi \in \cH$ is a unit vector and $(U \otimes I_n)(I_n \otimes V)=(I_n \otimes V)(U \otimes I_n)$.  By multiplying the unitary $U$ by some $\lambda \in \bT$ if necessary, we may assume that $$\text{Tr}(MX)=\omega_{qc}^*(G,X)=\sup \{|\text{Tr}(MZ)|:Z \in B_{qc}(n,n)\}.$$  
Since $B_{qc}(n,n)$ is convex \cite{HP}, it follows that $Y:=\frac{1}{2}(X+X^*) \in B_{qc}(n,n)$.  By Proposition \ref{selfadjointobservable}, $Y$ is represented by self-adjoint observables.  Finally, we see that 
\begin{align*}
\omega_{qc}^*(G,Y)&=\text{Tr} \left( \frac{1}{2}(X+X^*)M \right) \\
&=\frac{1}{2} (\text{Tr}(XM)+\text{Tr}(X^*M)) \\
&=\frac{1}{2}(\text{Tr}(XM)+\overline{\text{Tr}(XM)}) \\
&=\omega_{qc}^*(G),
\end{align*}
using the fact that $M$ is self-adjoint.  This establishes the reverse inequality for the commuting case.  For the entanglement bias, the same argument shows that if $X \in B_q(n,n)$ with $\omega_q(G,X)=\alpha \in [0,1]$, then $Y:=\frac{1}{2}(X+X^*) \in B_q(n,n)$ satisfies $\omega_q(G,Y)=\alpha$.  Using Proposition \ref{selfadjointobservable} and taking the supremum over all such strategies, the result holds for the entanglement bias.
\end{proof} 

\section{Main Results}

In this section, we connect the embedding problem with commuting and entanglement bias for quantum XOR games.  The first step is showing that, when considering Connes' embedding problem, it is enough to consider self-adjoint elements of $B_{qc}(m,m)$ and $B_{qa}(m,m)$ for all $m \geq 2$.  Lemma \ref{selfadjoint} allows for this reduction.

\begin{lem}
\label{selfadjoint}
Let $X \in M_n \otimes M_n$ and $t \in \{qa,qc\}$.  Then $X \in B_t(n,n)$ if and only if $$W:=\begin{pmatrix} 0 & 0 & 0 & X \\ 0 & 0 & 0 & 0 \\ 0 & 0 & 0 & 0 \\ X^* & 0 & 0 & 0 \end{pmatrix} \in B_t(2n,2n).$$
\end{lem}

\begin{proof}
Suppose that $t=qc$ and $X \in B_{qc}(n,n)$.  Then there are unitaries $U=(U_{ij}), \, V=(V_{k\ell}) \in M_n(\bofh)$ and a vector $\psi \in \cH$ of norm $1$ such that, for all $i,j,k,\ell$, we have $U_{ij}V_{k\ell}=V_{k\ell}U_{ij}$ and $$X_{(i,j),(k,\ell)}=\la U_{ij}V_{k\ell}\psi,\psi \ra.$$
Let $\widetilde{U}=\begin{pmatrix} 0 & (U_{ij}) \\ (U_{ji}^*) & 0 \end{pmatrix} \in M_{2n}(\bofh)$ and $\widetilde{V}=\begin{pmatrix} 0 & (V_{k\ell}) \\ (V_{\ell k}^*) & 0 \end{pmatrix} \in M_{2n}(\bofh)$, which are unitary.  The entries of $\widetilde{U}$ commute with the entries of $\widetilde{V}$, so with $\widetilde{U}$, $\widetilde{V}$ and $\psi$, we obtain $W' \in B_{qc}(2n,2n)$, where $$W'=\begin{pmatrix} 0 & 0 & 0 & X \\ 0 & 0 & \la U_{ij}V_{\ell k}^*\psi,\psi \ra & 0 \\ 0 & \la U_{ji}^*V_{k\ell} \psi,\psi \ra & 0 & 0 \\ X^* & 0 & 0 & 0 \end{pmatrix}.$$
With $Z=(U_{ij}V_{\ell k}^*\psi,\psi) \in B_{qc}(n,n)$, we see by Proposition \ref{starnorm} that $$W_Z:=W'=\begin{bmatrix} 0 & 0 & 0 & X \\ 0 & 0 & Z & 0 \\ 0 & Z^* & 0 & 0 \\ X^* & 0 & 0 & 0 \end{bmatrix} \in B_{qc}(2n,2n).$$
A similar argument using the unitaries $i (U_{ij})$ and $-i (V_{k\ell})$ shows that $W_{-Z} \in B_{qc}(2n,2n)$.  By convexity, we obtain $W=\frac{1}{2}(W_Z+W_{-Z}) \in B_{qc}(2n,2n)$.  If $t=qa$ and $\ee>0$, then there is $Y \in B_q(n,n)$ such that $|X_{(i,j),(k,\ell)}-Y_{(i,j),(k,\ell)}|<\ee$ for all $i,j,k,\ell$.  The above argument shows that $$R=\begin{bmatrix} 0 & 0 & 0 & Y \\ 0 & 0 & 0 & 0 \\ 0 & 0 & 0 & 0 \\ Y^* & 0 & 0 & 0 \end{bmatrix} \in B_q(2n,2n),$$
and $|R_{(a,b),(c,d)}-W_{(a,b),(c,d)}|<\ee$ for all $1 \leq a,b,c,d \leq 2n$.  Since $B_{qa}(2n,2n)=\overline{B_q(2n,2n)}$, we see that $W \in B_{qa}(2n,2n)$.

Conversely, suppose that $$W=\begin{bmatrix} 0 & 0 & 0 & X \\ 0 & 0 & 0 & 0 \\ 0 & 0 & 0 & 0 \\ X^* & 0 & 0 & 0 \end{bmatrix}$$
is in $B_q(2n,2n)$. Let $\cH_A$ and $\cH_B$ be finite-dimensional Hilbert spaces, $U \in M_{2n}(\cB(\cH_A))$ and $V \in M_{2n}(\cB(\cH_B))$ be unitaries, and $\psi \in \cH_A \otimes \cH_B$ be a unit vector such that for all $1 \leq a,b,c,d \leq 2n$, we have $W_{(a,b),(c,d)}=\la U_{ab}\otimes V_{cd}\psi,\psi \ra$.  Let $$S=(U_{i,(j+n)})_{i,j=1}^n \in M_n(\cB(\cH_A)) \text{ and } T=(V_{k,n+\ell})_{k,\ell=1}^n \in M_n(\cB(\cH_B)).$$
Then $S$ and $T$ are contractions.  Applying the Halmos dilation and performing a canonical shuffle, we obtain unitaries $\widetilde{U} \in M_n(\cB(\cH_A^{(2)}))$ and $\widetilde{V} \in M_n(\cB(\cH_B^{(2)}))$, where $$\widetilde{U}_{ij}=\begin{bmatrix} S_{ij} & (\sqrt{I-S^*S})_{ij} \\ (\sqrt{I-SS^*})_{ij} & -S_{ji}^* \end{bmatrix}$$
and similarly $$\widetilde{V}_{k\ell}=\begin{bmatrix} T_{k\ell} & (\sqrt{I-T^*T})_{k\ell} \\ (\sqrt{I-TT^*})_{k\ell} & -T_{\ell k}^* \end{bmatrix}.$$
Taking $\widetilde{\psi}=\begin{bmatrix} \psi & 0 & 0 & 0 \end{bmatrix}^t \in \cH_A^{(2)} \otimes \cH_B^{(2)}$ gives $$X=(\la (S_{ij} \otimes T_{k\ell})\psi,\psi \ra)_{(i,j),(k,\ell)}=(\la (\widetilde{U}_{ij} \otimes \widetilde{V}_{k\ell})\widetilde{\psi},\widetilde{\psi} \ra)_{(i,j),(k,\ell)} \in B_q(n,n).$$
Since $B_q(m,m)$ is dense in $B_{qa}(m,m)$ for all $m \geq 2$, the converse follows for $t=qa$.

Finally, assume that $W \in B_{qc}(2n,2n)$, and let $U,V \in M_{2n}(\bofh)$ be unitaries and $\psi \in \cH$ be a unit vector such that $W_{(a,b),(c,d)}=\la U_{ab}V_{cd}\psi,\psi\ra$ for all $1 \leq a,b,c,d \leq 2n$.  As before, let $S=(U_{i,(j+n)})_{i,j=1}^n$ and $T=(V_{k,n+\ell})_{k,\ell=1}^n$. We use an argument similar to the proof of \cite[Proposition 4.6]{harris}.  First, we let $$C_{ij}=\begin{bmatrix} S_{ij} & 0 \\ 0 & S_{ij} \end{bmatrix} \in \cB(\cH^{(2)}) \text{ and } D_{k\ell}=\begin{bmatrix} T_{k\ell} & (\sqrt{I-T^*T})_{k\ell} \\ (\sqrt{I-TT^*})_{k\ell} & -T_{\ell k}^* \end{bmatrix} \in \cB(\cH^{(2)}).$$
Since the set $\{S_{ij},S_{ij}^*\}_{i,j=1}^n$ commutes with the set $\{T_{ij},T_{ij}^*\}_{i,j=1}^n$, it follows that, by examining polynomials in $T$ and $T^*$, the set $\{S_{ij},S_{ij}^*\}_{i,j=1}^n$ commutes with each entry of $D_{k\ell}$ and $D_{k\ell}^*$.  Therefore, the set $\{C_{ij},C_{ij}^*\}_{i,j=1}^n$ commutes with the set $\{D_{k\ell},D_{k\ell}^*\}_{k,\ell=1}^n$, while $C=(C_{ij})$ is a contraction and $D=(D_{k\ell})$ is a unitary.  Performing a similar dilation on $C$ and replacing $D_{k\ell}$ with $\begin{bmatrix} D_{k\ell} & 0 \\ 0 & D_{k\ell} \end{bmatrix}$, we obtain unitaries $A=(A_{ij})$ and $B=(B_{k\ell})$ in $M_n(\cB(\cH^{(4)}))$ such that the $(1,1)$-block of $A_{ij}$ is $S_{ij}$ and the $(1,1)$-block of $B_{k\ell}$ is $T_{k\ell}$.  Letting $\widetilde{\psi}=\begin{bmatrix} \psi & 0 & 0 & 0 \end{bmatrix}^t \in \cH^{(4)}$, we see that $$X=(\la A_{ij}B_{k\ell}\widetilde{\psi},\widetilde{\psi} \ra)_{(i,j),(k,\ell)} \in B_{qc}(n,n),$$
which completes the proof.
\end{proof}

The following theorem shows that it is enough to consider self-adjoint elements of $M_n \otimes M_n$ for the embedding problem.

\begin{mythe}
\label{hpconnes}
The following are equivalent.
\begin{enumerate}
\item
Connes' Embedding Problem has a positive answer.
\item
$B_{qa}(n,n)=B_{qc}(n,n)$ for all $n \geq 2$.
\item
$M_n \otimes_{qa} M_n=M_n \otimes_{qc} M_n$ isometrically for all $n \geq 2$.
\item
For every $n \geq 2$ and $X=X^* \in M_n \otimes M_n$ with $\|X\|_{qc}=\|X\|_{M_n \otimes_{\min} M_n}=1$, we have $\|X\|_{qa}=1$.
\item
For every $n \geq 2$ and $M=M^* \in M_n \otimes M_n$ such that $\|M\|_1=\|M\|_{qc}^*=1$, we have $\|M\|_{qa}^*=1$.
\end{enumerate}
\end{mythe}

\begin{proof}
The equivalence of (1), (2) and (3) is proven in \cite{HP}.  Clearly 3 implies 4 and 5.  We will show that (5) implies (2); the proof that (4) implies (2) is similar.  Let $A \in M_n \otimes M_n$.  By the proof that (2) implies (1), it suffices to know that the following holds for all $n \geq 2$: whenever $Y \in B_{qc}(n,n)$ is diagonal with diagonal entries $\tau(u_iu_j^*)$ for some unitaries $u_1,...,u_n$ in a unital $C^*$-algebra $\cA$ and a tracial state $\tau$ on $\cA$, we have that $Y \in B_{qa}(n,n)$.  In particular, there are entries in such $Y$ equal to $1$.  Therefore, if $\| \cdot \|$ denotes the operator norm on $M_n \otimes M_n$, then $$1 \leq \|Y\| \leq \|Y\|_{qc}=1.$$
Now, assume that (2) fails.  Then there is $Y \in B_{qc}(n,n)$ with operator norm $1$ such that $Y \not\in B_{qa}(n,n)$. By Lemma \ref{selfadjoint}, there is $n \geq 2$ and some $X=X^* \in M_{2n} \otimes M_{2n}$ such that $\|X\|_{qc}=\|X\|=1$ but $\|X\|_{qa}>1$.  By the Hahn-Banach Theorem, we have 
\begin{align*}
1&=\sup \{ |g(X)|: g \in (M_{2n} \otimes M_{2n})^*, \, \|g\|_{M_n \otimes_{\min} M_n}^*=1\} \\
&<\sup \{|f(X)|: f \in (M_{2n} \otimes M_{2n})^*, \, \|f\|_{qa}^*=1\}.
\end{align*}
Therefore, there is $g \in (M_{2n} \otimes M_{2n})^*$ such that $\|g\|_{qa}^*=1$ but $|g(X)|>1$.  A scaling argument shows that we may choose $\ee>0$ and $g \in (M_{2n} \otimes M_{2n})^*$ such that $\|g\|_{M_n \otimes_{\min} M_n}^*=1=|g(X)|$ and $\|g\|_{qa}^*=1-\ee$.  By multiplying $g$ by some $z \in \bT$ if necessary, we may assume that $g(X)=1$.  Let $M \in M_{2n} \otimes M_{2n}$ be the matrix such that $g(Y)=\text{Tr}(YM)$ for all $Y \in M_{2n} \otimes M_{2n}$.  Since $X=X^*$, we see that $$g^*(X)=\text{Tr}(XM^*)=\overline{\text{Tr}(MX)}=\overline{\text{Tr}(XM)}=\overline{g(X)}=1.$$
Since $\| \cdot \|_{qa}^*$ is a $*$-norm, we obtain $\|g^*\|_{qa}^* \leq 1-\ee$ and $\|g^*\|_{M_n \otimes_{\min} M_n}^*=1=g^*(X)$.  Thus, $f=\text{Re}(g):=\frac{g+g^*}{2}$ is a functional with $\|f\|_{qa}^* \leq 1-\ee$ and $f(X)=1=\|f\|_{M_n \otimes_{\min} M_n}^*$.  The associated matrix to $f$ is $\frac{M+M^*}{2}$, which is self-adjoint.  By the contrapositive, (5) implies (2).
\end{proof}

As a corollary, we can describe the embedding problem in terms of perfect strategies for quantum XOR games.

\begin{cor}
\label{winning}
The following are equivalent.
\begin{enumerate}
\item
Connes' embedding problem has a positive answer.
\item
For every $n \geq 2$ and for every quantum XOR game $G$ of size $n$ with associated matrix $M \in M_n \otimes M_n$, we have $\omega_{qa}^*(G)=\omega_{qc}^*(G)$.
\item
Whenever $n \geq 2$ and $G$ is a quantum XOR game of size $n$ with a perfect $qc$-strategy, there is also a perfect $qa$-strategy for $G$.
\end{enumerate}
\end{cor}

\begin{proof}
For $t \in \{qa,qc\}$, the quantity $\omega_t^*(G)$ is the norm of a self-adjoint linear functional on $M_n \otimes_t M_n$.  In particular, if Connes' embedding problem holds, then by Theorem \ref{hpconnes}, $B_{qa}(n,n)=B_{qc}(n,n)$ for all $n$, so that $\omega_{qa}^*(G)=\omega_{qc}^*(G)$ for all quantum XOR games $G$.  Hence, (1) implies (2).  Clearly (2) implies (3), so it remains to show that (3) implies (1).  If (1) fails, then by Theorem \ref{hpconnes}, there is some $n \geq 2$, $M=M^* \in M_n \otimes M_n$, and $0<\ee<1$ such that $$1-\ee=\|M\|_{qa}^*<\|M\|_{qc}^*=\|M\|_{M_n \otimes_{\min} M_n}^*=1.$$
This implies that $\|M\|_1=1$.  Then there is a quantum XOR game $G$ with associated matrix $M$ \cite{RV}.  By the choice of $M$, for every correlation $X=(X_{(i,j),(k,\ell)})_{(i,j),(k,\ell)} \in B_{qa}(n,n)$, we have $$\omega_{qa}^*(G;X) \leq 1-\ee.$$
Hence, there is no perfect $qa$-strategy for $G$.  Meanwhile, there is $Y \in B_{qc}(n,n)$ with $\omega_{qc}^*(G;Y)=1$, so that $Y$ is a perfect $qc$-strategy for $G$.  This completes the proof.
\end{proof}

\end{document}